\documentclass{amsart}
\usepackage{amsmath, amsfonts, amssymb, url, mathtools, verbatim}

\newcommand{\mA}{\mathcal{A}}

\newcommand{\mC}{\mathcal{C}}

\newcommand{\mE}{\mathcal{E}}

\newcommand{\mJ}{\mathcal{J}}
\newcommand{\mK}{\mathcal{K}}
\newcommand{\mL}{\mathcal{L}}

\newcommand{\mS}{\mathcal{S}}

\newcommand{\fm}{\mathfrak{m}}

\newcommand{\fP}{\mathfrak{P}}

\newcommand{\bfF}{\mathbf{F}}
\newcommand{\bfG}{\mathbf{G}}

\newcommand{\bfQ}{\mathbf{Q}}

\newcommand{\bfT}{\mathbf{T}}

\newcommand{\bfZ}{\mathbf{Z}}

\newcommand{\ov}{\overline}

\newcommand{\be}{\begin{equation}}
\newcommand{\ee}{\end{equation}}
\newcommand{\bes}{\begin{equation*}}
\newcommand{\ees}{\end{equation*}}

\newcommand{\bs}{\begin{split}}
\newcommand{\es}{\end{split}}
\newcommand{\bss}{\begin{split*}}
\newcommand{\ess}{\end{split*}}

\newcommand{\bmat}{\left[ \begin{matrix}}
\newcommand{\emat}{\end{matrix} \right]}
\newcommand{\bsmat}{\left[ \begin{smallmatrix}}
\newcommand{\esmat}{\end{smallmatrix} \right]}

\newcommand{\bml}{\begin{multline}}
\newcommand{\eml}{\end{multline}}
\newcommand{\bmls}{\begin{multline*}}
\newcommand{\emls}{\end{multline*}}

\newcommand{\new}{\mathrm{new}}

\DeclareMathOperator{\Cl}{Cl}

\DeclareMathOperator{\Frob}{Frob}
\DeclareMathOperator{\Gal}{Gal}
\DeclareMathOperator{\GL}{GL}

\DeclareMathOperator{\Hom}{Hom}

\DeclareMathOperator{\Tr}{Tr}

\DeclareMathOperator{\val}{val}

\newcommand{\tr}{\textup{tr}\hspace{2pt}}

\newcommand{\T}{\mathbf{T}}

\usepackage[all]{xy}
\usepackage{amsthm}
\usepackage{amssymb, amsfonts}
\SelectTips{cm}{10}\UseTips
\theoremstyle{plain}
\newtheorem{thm}{Theorem}
\newtheorem{prop}[thm]{Proposition}
\newtheorem{cor}[thm]{Corollary}
\newtheorem{lemma}[thm]{Lemma}

\theoremstyle{definition}

\newtheorem{example}[thm]{Example}
\newtheorem{rem}[thm]{Remark}

\numberwithin{thm}{section}
\numberwithin{equation}{section}

\author{Krzysztof Klosin}
\address{Department of Mathematics, 
	Queens College, 
	City University of New York, 
	65-30 Kissena Blvd
	Flushing, NY 11367, USA}
\email{kklosin@qc.cuny.edu} 

\author{Mihran Papikian}
\address{Department of Mathematics, Pennsylvania State University, University Park, PA 16802, USA}
\email{papikian@psu.edu}

\begin{document}
\title{Galois Extensions and a Conjecture of Ogg}
\thanks{The first author's research was supported by a Collaboration for Mathematicians Grant \#578231  from the Simons Foundation
 and by a PSC-CUNY research award jointly funded by the Professional Staff Congress and the City
University of New York.}


\begin{abstract} 
	Let  $N=pq$ be a product of two distinct primes.  
	There is an isogeny $J_0(N)^\new\to J^N$ defined over $\bfQ$ between the new quotient of 
	$J_0(N)$ and the Jacobian of the Shimura curve attached to the indefinite quaternion algebra of discriminant $N$. 
	In the case when $p=2,3,5,7,13$, Ogg made predictions about the kernels of these isogenies. 
	We show that Ogg's conjecture is not true in general. Afterwards, we 
	propose a strategy for proving results toward Ogg's conjecture in certain situations. Finally,  
	we discuss this strategy in detail for $N=5\cdot 13$. 
\end{abstract}

\date{\today}

\maketitle

\section{Introduction}\label{Intro}  

\subsection{Ogg's conjecture} Let $N$ be a product of an even number of distinct primes. Let $J_0(N)$ 
be the Jacobian of the modular curve $X_0(N)$. In \cite{Ribet80}, 
Ribet proved the existence of an isogeny defined over $\bfQ$ between the ``new'' 
part $J_0(N)^\new$ of $J_0(N)$ and the Jacobian $J^N$ 
of the Shimura curve $X^N$ attached to a maximal order in the indefinite quaternion algebra over $\bfQ$ of discriminant $N$. 
The proof proceeds by showing that the $\bfQ_\ell$-adic Tate modules of $J_0(N)^\new$ and $J^N$ are 
isomorphic as $\Gal(\overline{\bfQ}/\bfQ)$-modules, which is a 
consequence of a correspondence between automorphic forms on $\GL(2)$ and automorphic forms on the multiplicative group of a 
quaternion algebra. The existence of the isogeny $J_0(N)^\new\to J^N$ defined 
over $\bfQ$ then follows from a special case of Tate's isogeny conjecture for abelian varieties over number fields,  
also proved in \cite{Ribet80} (the general case of Tate's conjecture was proved a few years later by Faltings). 
Unfortunately, this argument provides no information about the isogenies $J_0(N)^\new\to J^N$ 
beyond their existence. 

In \cite{Ogg85}, Ogg made explicit predictions about the kernel of Ribet's 
isogeny when $N=pq$ is a product of two distinct primes and $p=2,3,5,7,13$. 
In this case, $J_0(N)^\new$ 
is the quotient of $J_0(N)$  by the subvariety generated by the images of $J_0(q)$ in $J_0(N)$ under the maps 
induced by the two degeneracy morphisms $X_0(pq)\rightrightarrows X_0(q)$ (note that $J_0(p)=0$). 
Let $\mC$ be the cuspidal divisor group of $J_0(N)$, which is well-known to be a finite abelian subgroup of $J_0(N)(\bfQ)$; we 
refer to \cite{ChuaLing97} for a complete description of $\mC$.  Let 
$\overline{\mC}$ be the image of $\mC$ in $J_0(N)^\new$.  Denote 
$$
M= \text{numerator of }(q+1)/12.
$$
Ogg's conjecture predicts that there is an isogeny 
$J_0(N)^\new\to J^N$ whose kernel $\mK$ is a subgroup of $\overline{\mC}$ such that 
\begin{align}\label{eqOggConj}
\nonumber	\mK &\cong \bfZ/M \quad \text{for} \quad p=2,3,5, \\ 
	\mK &\cong \bfZ/2M \quad \text{for} \quad p=7,  \\
\nonumber	\mK &\cong \bfZ/7\oplus \bfZ/M \quad \text{for}\quad p=13. 
\end{align}

The underlying idea behind Ogg's conjecture is to compare the component groups of the N\'eron models 
of $J_0(N)$ and $J^N$ at $q$, which provides some reasonable guesses for the kernels of Ribet's isogenies. In fact, Ogg imposes the 
restriction $p=2,3,5,7,13$ to be able to carry out the necessary calculations. We briefly sketch Ogg's reasoning. 
For simplicity, we ignore the $2$ and $3$-primary torsion of the groups involved in the discussion, and also the case $p=13$ 
where $\mK$ might not be cyclic. 
The component groups of $J_0(N)$ for square-free $N$ 
are relatively easy to describe; cf. \cite[Appendix]{Mazur78}. 
On the other hand, although the component groups of $J^N$ can be computed for a given $N$ by combining 
a classical method of Raynaud  with a result of Cherednik and Drinfeld about the reduction of $X^N$ at $q$, 
these groups do not exhibit any regular patterns so cannot be described using only the prime decomposition of $N$ (as is the case 
for $J_0(N)$). One exception is the case when $N=pq$ and  $p=2,3,5,7,13$. In this case (and only in this case), 
the dual graph of the special fibre of the Cherednik-Drinfeld model of $X^N$ at $q$ has two vertices, so the component group 
is easy to compute and turns out to be cyclic of order $(q+1)$. The 
component group $\Phi_q$ of $J_0(N)$ at $q$ is cyclic of order $(q-1)$. 
Next, Ogg considers the canonical specialization $\mC\to \Phi_q$, and shows that the ``old'' part of $\mC$ arising from the 
cuspidal divisor group of $J_0(q)$ maps surjectively onto $\Phi_q$, whereas a specific ``new'' cuspidal divisor $D$ of order $q+1$ 
maps to $0$ in $\Phi_q$.  Then the kernel $\mK$ in \eqref{eqOggConj} is predicted to be generated by the image of $D$ in $J_0(N)^\new$. 
The fact that $J_0(N)^\new$ and $J^N$ have purely toric reduction at $q$ is implicitly used in this last step. 
(Given an abelian variety $A$ over a local field $K$ with purely toric reduction and a 
finite constant subgroup $H\subset A(K)$, it is possible to describe the component group of $A/H$ in terms 
of the component group $\Phi_A$ of $A$ and the kernel/image of the canonical specialization $H\to \Phi_A$; 
cf. \cite[Thm. 4.3]{PapikianJNT2011}.)

Let $\bfT\subset \mathrm{End}(J_0(N))$ be the Hecke algebra generated over $\bfZ$ by all Hecke correspondences $T_\ell$ 
with prime indices  (including those that divide $N$). The ring ${\bfT}$ also acts on $J_0(N)^\new$ and $J^N$ (cf. \cite{Ribet90}), and 
it is implicit in \cite{Ribet80} that there is an isogeny $J_0(N)^\new\to J^N$ over $\bfQ$ which is ${\bfT}$-equivariant (cf. \cite[Cor. 2.4]{Helm07}). 
Since the cuspidal divisor group $\mC$ is annihilated by the Eisenstein ideal of $\bfT$, Ogg's conjecture 
implies that, in the case when $N=pq$ and $p=2,3,5,7,13$, there is an isogeny $J_0(N)^\new\to J^N$ whose 
kernel is supported on the (new) Eisenstein maximal ideals. 
(The Eisenstein ideal $\mE$ of $\bfT$ is the ideal generated by all $T_\ell-(\ell+1)$ for primes $\ell \nmid N$; 
the Eisenstein maximal ideals are the maximal ideals containing $\mE$.) 

In \cite{Ribet90Israel}, Ribet proved a theorem which implies that 
the support of the kernel of a ${\bfT}$-equivariant isogeny $J_0(pq)^\new\to J^{pq}$ must, in general, contain maximal 
ideals of ${\bfT}$ which are not Eisenstein, so any construction of such an isogeny must be relatively elaborate. He then  
gave a concrete example with $p=11, q=193$  where this phenomenon occurs.  
Next, we show that Ribet's construction can be carried out also in some cases when $p=2,3,5,7,13$; 
thus Ogg's conjecture \eqref{eqOggConj} is not true in general\footnote{To be fair, Ogg writes in his paper \cite[p. 213]{Ogg85} 
	``On devine (deviner est plus faible que conjecturer) donc qu'on peut prendre $\mK$ comme noyau de l'isog\'enie, avec 
	confiance si la partie ancienne est triviale, i.e. si $(q-1)\mid 12$." Hence, perhaps, we should have called \eqref{eqOggConj} 
``Ogg's guess''.}. 

\begin{example} 
	Let $E$ be the elliptic curve over $\bfQ$ defined by the equation $$y^2+y=x^3-x^2-2x+1.$$ This is the 
	unique, up to isomorphism, elliptic curve of conductor $q=701$ (which is a prime); cf. \cite{Cremona97}. 
	In particular, $E$ has no cyclic isogenies defined over $\bfQ$, so $E[3]$ is an irreducible 
	$\mathrm{Gal}(\overline{\bfQ}/\bfQ)$-module. Let $\rho: \mathrm{Gal}(\overline{\bfQ}/\bfQ) \to \mathrm{Aut}(E[3])$ 
	be the corresponding Galois representation.  
	Put $p=7$. By \cite[Thm. 1]{Ribet90Israel}, there is a maximal ideal $\fm\lhd \bfT$ of residue characteristic $3$ such that 
	the kernel $J_0(pq)[\fm]$ of $\fm$ on $J_0(pq)(\overline{\bfQ})$ defines a representation equivalent to $\rho$. 
	One easily checks either by hand, or with the help of \texttt{Magma}, that $E(\bfF_p)\cong \bfZ/3\times \bfZ/3$. This implies that 
$\rho(\mathrm{Frob}_p)=1$. In particular, 
$$
\Tr(\rho(\mathrm{Frob}_p))=2\equiv (p+1)\mod \fm. 
$$
As is explained in \cite{Ribet90Israel}, the above congruence implies that $\fm$ is new. By Theorem 2 in \cite{Ribet90Israel}, 
$\dim_{\bfF_3}J_0(pq)^\new[\fm]=2$. 
On the other hand, since $\rho$ is unramified at $p$, Theorem 3 in \cite{Ribet90Israel} 
applies, so $\dim_{\bfF_3}J^{pq}[\fm]=4$. 
It easily follows from this that the kernel of any ${\bfT}$-equivariant isogeny $J_0(pq)^\new\to J^{pq}$ must have the non-Eisenstein $\fm$ in its support, 
contrary to Ogg's conjecture. 
(Otherwise, by duality, there is a homomorphism $J^{pq} \to J_0(pq)$ with finite kernel whose support does not contain $\fm$.  
This implies that there is an injection $J^{pq}[\fm] \to J_0(pq)[\fm]$, which 
is absurd.)

A similar construction also works for $p=13$ and $q=571$. Let 
$E: y^2+y=x^3+x^2-4x+2$ be the curve $571$ B1 in Cremona's table \cite{Cremona97}. Again, $E[3]$ is irreducible and the corresponding 
Galois representation $\rho$ satisfies $\rho(\mathrm{Frob}_{13})=-1$. Ribet's theorems then imply that $\dim_{\bfF_3}J^{pq}[\fm]=4$ 
and $\dim_{\bfF_3}J_0(pq)^\new[\fm]=2$, from which one obtains a contradiction to \eqref{eqOggConj} as before.  
\end{example}

Despite the fact that Ogg's conjecture is false in general, some cases of the 
conjecture for \textit{small} levels have been proved. The conjecture is easy to verify when 
$J_0(pq)^\new$ and $J^{pq}$ are elliptic curves (there are five such cases). When 
$X^{pq}$ is hyperelliptic of genus $2$ or $3$, Ogg's conjecture is verified in \cite{GonzalezRotger04} and 
\cite{GonzalezMolina16} (there are twelve such cases). The strategy here is 
to explicitly compute and compare the period matrices of $J_0(pq)^\new$ and $J^{pq}$, which itself relies 
on a lengthy calculation of the defining equations of hyperelliptic Shimura curves. 
When $N=5\cdot 13$, Ogg's conjecture is verified in \cite{KlosinPapikian18}, up to $2$-primary torsion supported on 
a maximal Eisenstein ideal. In this case, $X^N$ has genus $5$ and is not hyperelliptic. Our approach in 
\cite{KlosinPapikian18} is completely different from \cite{GonzalezRotger04, GonzalezMolina16} and relies 
on the Hecke equivariance of Ribet isogenies and the fact that  the Hecke algebra of level $65$ is a rather simple ring. 

For general $N=pq$, Yoo \cite{YooBLMS} proved that,  
under certain congruence assumptions on $p$, $q$, and $\ell$,  
the kernel of a Ribet isogeny $J_0(N)^\new\to J^N$ must contain the $\ell$-primary subgroup of the cuspidal divisor group $\overline{\mC}$. 
This result implies that for $p=2,3,5,7,13$ and odd $\ell\geq 5$, $\ker(\pi)$ contains $\mK\otimes\bfZ_\ell$ from \eqref{eqOggConj}, in 
accordance with Ogg's conjecture. 

\subsection{Main result} 
In this article we continue exploring avenues that lead to partial results toward Ogg's conjecture.  
While we again employ the Hecke algebra, we propose a different approach from \cite{KlosinPapikian18} 
which has the advantage of being applicable to larger values of $N$ than $65$. 

Now we outline our approach and state the main results.  To simplify the notation, let $J:=J_0(N)$ and $J':=J^N$. 
Let $S$ denote the finite set of maximal ideals of $\bfT$ that are either Eisenstein, or of residue characteristic $2$ or $3$. 
There is an element $\sigma_S\in \bfT$ such that for any maximal ideal $\fm$ of $\bfT$, one has $\sigma_S \in \fm$ if and only if $\fm \in S$ (cf. Lemma 3.2 in \cite{Helm07}). Set $\bfT_S:= \bfT[\sigma_S^{-1}]$. 

$J^{\rm new}$ and $J'$ have purely toric reduction at the primes $p$ and $q$, and good reduction everywhere else.  
For $A=J^{\rm new}$ or $J'$, denote by $M_p(A)=\Hom(\mA^0_{\ov{\bfF}_p}, \bfG_{m, \ov{\bfF}_p})$ the character group of  $A$ at $p$.   
Here $\mA$ is the N\'eron model of $A$ over $\bfZ_p$, and $\mA_{\bfF_p}^0$ is the connected component 
of the identity of the special fibre of $\mA$ at $p$. The character group $M_p(A)$ 
is a free abelian group of rank equal to $\dim(A)$. We similarly define the character group $M_q(A)$ at $q$. 
By the N\'eron mapping property, $\bfT$ acts on $M_p(A)$ and $M_q(A)$. 

A special case of a result of Helm \cite[Prop. 8.13]{Helm07} implies that there is an isomorphism of $\bfT_S$-modules 
\begin{equation}\label{eqHelm} 
\Hom(J^{\rm new}, J') \cong_{\bfT_S} \Hom(M_q(J'), M_q(J^{\rm new})). 
\end{equation}
On the other hand, a special case of a result of Ribet \cite[Thm. 4.1]{Ribet90} implies that 
\begin{equation}\label{eqRibet} 
M_q(J')\cong_{\bfT} M_p(J^{\rm new}). 
\end{equation}
Since the cuspidal divisor group of $J$ is annihilated by the Eisenstein ideal of $\bfT$, \eqref{eqOggConj} combined with 
\eqref{eqHelm} and \eqref{eqRibet} 
implies that 
\begin{equation}\label{eq1.3}
	M_p(J^{\rm new}) \cong_{\bfT_S} M_q(J^{\rm new}). 
\end{equation} 
Conversely, if \eqref{eq1.3} is true, then \eqref{eqHelm} and \eqref{eqRibet}  imply that there 
is an isogeny $\pi: J^{\rm new} \to J'$ whose kernel is supported on the maximal ideals in $S$.

This offers a natural strategy for proving results toward \eqref{eqOggConj}. First, one needs to prove \eqref{eq1.3}. Since the 
character groups are free $\bfZ$-modules, 
this step involves only  linear algebra  calculations, which may be quite daunting in practice -  but we note here that there exist algorithms that allow one to do this at least in principle; cf. section \ref{Matrices}. The second step comprises classifying isogenies supported on the 
maximal ideals in $S$. This can be achieved by excluding  the existence of certain subgroup schemes in $J[\fm^s]$ for $\fm \in S$, a problem which in \cite{KlosinPapikian18} (for $N=65$) was handled by an ad hoc counting argument.

In this paper we offer a more systematic approach for step 2 based on the non-existence of certain deformations of non-split Galois extensions 
\be \label{ext1} 0 \to \bfZ/{\ell} \to \ov{\rho} \to \mu_{\ell} \to 0,\ee 
where $\ell\geq 5$ is a prime. 
By the results of Ohta and Yoo \cite{Ohta14, Yoo16}, one knows that the residue characteristic of an Eisenstein maximal ideal 
divides either $p \pm 1$ or $q \pm 1$. 
We will assume that $\ell$ satisfies one of the following conditions:
\begin{align}
\label{eql1}\ell\mid (p+1) \text{ and } \ell\nmid (q\pm 1),  \\ 
\label{eql2} \ell\mid (q+1)\text{ and } \ell\nmid (p\pm 1). 
\end{align}
Put $\fm=(T_p+1, T_q-1, \mE, \ell)$ in the first case, and $\fm=(T_p-1, T_q+1, \mE, \ell)$ in the second case. 
Then $\fm$ is a new Eisenstein maximal ideal of residue characteristic $\ell$ and $\dim_{\bfF_{\ell}} J[\fm] = 2$; cf. \cite{Yoo16}, \cite{YooTAMS}. 
 In particular, 
 the action of $G_{\bfQ} = \Gal(\ov{\bfQ}/\bfQ)$ on $J[\fm]$ gives rise to an extension 
 \be \label{ext2} 0 \to \bfZ/{\ell} \to J[\fm] \to \mu_{\ell} \to 0.\ee 
This extension does not split. Indeed, by a theorem of Vatsal \cite{Vatsal05}, the extension \eqref{ext2} splits if and only if $\mu_{\ell} \subset \mS$, where $\mS$ denotes the Shimura subgroup of $J$. Ignoring the $2$ and $3$-primary torsion, one has $\# \mS = (p-1)(q-1)$; cf. \cite{LingOesterle91}.
 Thus for $\ell \nmid (p-1)(q-1)$ we see that $\mu_{\ell} \not\subset \mS$. Hence \eqref{ext2} can in fact be viewed as a non-split extension of Galois modules of the form \eqref{ext1}.  We also note that, ignoring the $2$ and $3$-primary torsion, the cuspidal divisor group $\mC$ of $J$ 
 and the Eisenstein ideal $\mE$ satisfy (cf. \cite{ChuaLing97}, \cite{Ohta14}, \cite{Yoo16})
 $$\bfT/\mE\cong \mC \cong \bfZ/(p-1)(q-1) \oplus \bfZ/(p+1)(q-1) \oplus \bfZ/(p-1)(q+1).$$
 This implies that $\fm$ is the unique Eisenstein maximal ideal of residue characteristic $\ell$ and the 
 constant subgroup scheme of $J[\fm]$ in \eqref{ext2} is $\mC[\ell]$. 
 
In Theorem \ref{main3} (and Corollary \ref{main4}) we prove that under the above assumptions on $\ell$, the Galois representation $\ov{\rho}:= J[\fm]$ does not admit any (non-trivial) reducible (Fontaine-Laffaille) deformations  of determinant $\epsilon$, the $\ell$-adic cyclotomic character (or its mod $\ell^m$ reduction). This allows us to prove the following result, which is the main theorem of the paper.  
 
\begin{thm} \label{thmintro} 
	Assume \eqref{eq1.3} is satisfied, so that there is an isogeny $\pi: J^\new\to J'$ with kernel supported 
	on the maximal ideals in $S$. Assume $\pi$ is chosen to have minimal degree. Let $\ell\geq 5$ be a prime that 
	satisfies either \eqref{eql1} or \eqref{eql2}. Let $\fm\in S$ be the new Eisenstein maximal ideal of residue characteristic $\ell$. 
	Assume further that $J^\new/J^\new[\fm] \cong J^\new$. 
Then the $\ell$-primary part of $\ker \pi$ is contained in $\overline{\mC}[\ell]\cong \bfZ/\ell$.
\end{thm} 
 \begin{proof}
Let $H$ be the $\ell$-primary part of $\ker(\pi)$.  
Note that $J^\new[\fm] \not\subset H$, since otherwise $\pi$ factors through 
$$J^\new\to J^\new/J^\new[\fm]\cong J^\new\overset{\pi'}{\to} J',$$ 
contradicting the minimality of the degree of $\pi$. Since $\fm$ is new and satisfies multiplicity one, we have $J[\fm]\cong J^\new[\fm]$. 
One can consider 
$H $ as a subgroup scheme of $J[\fm^s]$ for some $s \in \bfZ_+$. We claim 
that $H$ is a proper subscheme of $J[\fm]$. 
If this is not the case (i.e., $H$ is not  a proper subscheme of $J[\fm]$) then we see as in the proof of Proposition 4.5 in \cite{KlosinPapikian18} that without loss of generality we may assume that   $s=2$. The equivalence of (1) and (2) in Lemma 15.1 of \cite{Mazur78} implies that since $\dim_{\bfF_{\ell}} J[\fm]=2$ we get $J[\fm^2] \cong \bfT/\fm^2 \oplus \bfT/\fm^2$ as $\bfT$-modules. Hence $H=\bfT/\fm^{s_1} \oplus \bfT/\fm^{s_2}$, with $0 \leq s_1 \leq s_2$. Clearly $s_1=0$ since otherwise $H\supset H[\fm] = J[\fm]$. Also $s_2=2$ as otherwise $H \subset J[\fm]$. Hence $H$ is a Galois stable line (free $\bfT/\fm^2$-module of rank 1) in $J[\fm^2]$. Let $\chi_1$ be the character by which $G_{\bfQ}$ acts on this line and write $\chi_2$ for the character by which it acts on the quotient $J[\fm^2]/H$. Then the Galois representation $\rho: G_{\bfQ}\to \GL_2(\bfT/\fm^2)$ afforded by $J[\fm^2]$ satisfies the conditions in Corollary \ref{main4} with 
 $\{\ell_1, \ell_2\} = \{p,q\}$, 
$\Sigma'=\{p,q,\ell\}$ (we note that $\rho$ is in the image of the Fontaine-Laffaille functor since it arises as a subquotient of the Galois representation afforded by the Tate module of an abelian variety), so it cannot exist. Thus, $H\subsetneq J[\fm]$. Finally, because $J[\fm]$ 
is non-split, 
the only $G_\bfQ$-stable subgroup of $J[\fm]$ is its constant subgroup $\bfZ/\ell$ which comes from the cuspidal divisor group. 
\end{proof}

To conclude the introduction, let us briefly comment on how Theorem \ref{thmintro} applies to Ogg's conjecture. 
Assumption \eqref{eq1.3} can be checked using an explicit matrix representation of generators of $\bfT$.  
In the case $N=65$ we carry out this calculation in section \ref{Matrices}. In fact in this case we are able to prove a stronger result, 
namely that $M_p(J) \cong M_q(J)$ as $\bfT$-modules without inverting $\sigma_S$. 
(This also shows that \eqref{eqHelm} is not true without inverting the Eisenstein maximal ideals 
since the Jacobians $J$ and $J'$ are not isomorphic in this case.)
The assumption $J^\new/J^\new[\fm]\cong J^\new$ is satisfied if, for example, $\fm$ is globally principal in $\bfT$. 
(Indeed, if $\eta$ is a generator then $J^\new[\fm]$ is the kernel of the isogeny $J^\new \xrightarrow{\cdot \eta} J^\new$.) 
This stronger assumption on the principality of $\fm$ is satisfied for some Eisenstein maximal ideals in Hecke 
algebras of small levels, for example, those $N=pq$ for which $J=J^\new$ (equiv. $(p-1)$ and $(q-1)$ divide $12$), which is related to 
the fact that in those cases the Hecke algebra $\bfT\otimes \bfQ$ turns out to be a direct product of number fields of class number $1$. 
Finally, once we know $\ker(\pi)\otimes \bfZ_\ell\subset \overline{\mC}[\ell]$, the $\ell$-primary part 
of $\ker(\pi)$ can be determined by comparing the component groups of $J^\new$ and $J^\new/\overline{\mC}[\ell]$ 
with the component groups of $J'$, as was originally done by Ogg. 
In the case $N=65$, the prime $\ell=7$ is the only one which satisfies the conditions of 
Theorem \ref{thmintro} (the other two Eisenstein primes are 2 and 3). 
Thus Theorem \ref{thmintro} gives an alternative proof that for $N=65$ there is an isogeny $\pi: J\to J'$ such that 
$\ker(\pi)\otimes \bfZ_\ell\cong \mK\otimes \bfZ_\ell$ for $\mK$ in \eqref{eqOggConj} and $\ell\geq 5$.

\section{Non-existence of certain Galois extensions} \label{The Theorem}
Let $\ell>2$ and $\Sigma:=\{\ell_1, \ell_2, \dots, \ell_k\}$ be a set of distinct primes such that $\ell \nmid \ell_i(\ell_i-1)$ for $i=1, \dots, k$.  Write  $\Sigma':=\Sigma \cup \{\ell\}$ and $G_{\Sigma'}$ for the absolute Galois group of the maximal Galois extension of $\bfQ$ unramified outside $\Sigma'$.

Consider a representation $\ov{\rho}: G_{\Sigma'} \to \GL_2(\bfF_{\ell})$ which  is a non-split extension of the form $$\ov{\rho}= \bmat 1 & * \\ & \ov{\epsilon}\emat,$$ where $\epsilon$ will denote the $\ell$-adic cyclotomic character (but we will also use $\epsilon$ to denote the reduction of the $\ell$-adic cyclotomic character mod $\ell^m$) and $\ov{\epsilon}$ its mod $\ell$ reduction.

The main result of this section is Theorem \ref{main3} (and Corollary \ref{main4}) which asserts the non-existence of certain trace-reducible deformations of $\ov{\rho}$. The proof essentially boils down to showing that there are no (trace-reducible) deformations to $\bfZ/\ell^2$ as well as no non-trivial (trace-reducible) deformations to the dual numbers $\bfF_{\ell}[X]/X^2$. We begin with the $\bfZ/\ell^2$-case -- the harder of the two (Proposition \ref{main} below), which we prove in a slightly greater generality than needed for our purposes.
We fix once and for all an embedding $\ov{\bfQ} \hookrightarrow \ov{\bfQ}_{\ell}$.   Let $m\geq 2$ be an integer.
\begin{prop} \label{main} Suppose $\val_{\ell}(\ell_1^2-1)=m-1$ (which is equivalent to $\val_{\ell}(\ell_1+1)=m-1$) and $\val_{\ell}(\ell_i^2-1)=0$ (equivalent to $\val_{\ell}(\ell_i+1)=0$)  for all $i=2,3,\dots, k$. Then there does not exist a Galois representation $\rho_m: G_{\Sigma'}\to \GL_2(\bfZ/\ell^m)$ such that \begin{itemize}
\item [(i)] $\rho_m$ is crystalline in the image of the Fontaine-Laffaille functor at $\ell$; 
\item [(ii)] $\det \rho_m=\epsilon$;
\item [(iii)] $\tr \rho_m = \chi_1 + \chi_2$ for some Galois characters $\chi_1, \chi_2: G_{\Sigma'} \to (\bfZ/\ell^m)^{\times}$ with $\chi_1\equiv 1$ (mod $\ell$) and $\chi_2\equiv \ov{\epsilon}$ (mod $\ell$);
\item[(iv)] $\rho_m \equiv \ov{\rho}$ mod $\ell$.
\end{itemize}
\end{prop}
\begin{rem} Below for brevity we will refer to representations in the image of the Fontaine-Laffaille functor simply as Fontaine-Laffaille representations. All the properties of such representations that we will use are stated e.g. in   \cite{BergerKlosin13}.  \end{rem}
We prepare the proof of Proposition \ref{main} by several lemmas.

\begin{lemma} \label{4} We must have $\chi_1=1$ and $\chi_2=\epsilon$ \end{lemma}
\begin{proof} It is enough to show that $\chi_1=1$ as then $\chi_2=\epsilon$ by (ii). First note that since $\rho_m$ is a Fontaine-Laffaille representation and the category of these is closed under taking subquotients, so is $\chi_1$. Furthermore, $\chi_1$ is unramified outside $\Sigma'$. Hence to prove the claim it is enough to  show that the trivial character does not admit any non-trivial Fontaine-Laffaille infinitesimal deformations $\psi: G_{\Sigma'} \to (\bfF[X]/X^2)^{\times}$. This in turn can be proven as Proposition 9.5 of \cite{BergerKlosin13}. \end{proof}

To prove Proposition \ref{main} let us first note that by the main Theorem of \cite{Urban99} if $\rho_m$ whose trace splits as in (iii) exists then it can be conjugated to an upper-triangular representation of the form $$\rho_m = \bmat\chi_1 & * \\ & \chi_2\emat.$$ We can treat $\rho_m$ as an element  of $H^1(\bfQ, (\bfZ/\ell^m)(\chi_1\chi_2^{-1}))$ which does not lie in $H^1(\bfQ, (\ell\bfZ/\ell^m \bfZ)(\chi_1\chi_2^{-1}))$, i.e., is of maximal order. This is so, because the extension given by $\rho_m$ reduces mod $\ell$ to $\ov{\rho}$ which is not split.

For the moment we will work in a slightly greater generality and assume that $\chi_1=1$ and  $\chi_2=\chi = \epsilon^n$ for $n\neq 0$, however we apply it only in the case when $n=1$.  Set $$T=\bfZ_{\ell}(-n) = \bfZ_{\ell}(\epsilon^{-n}), \quad V=\bfQ_{\ell}(-n), \quad W=\bfQ_{\ell}/\bfZ_{\ell}(-n)$$ and $$ W_M:= \ell^{-M}\bfZ_{\ell}/\bfZ_{\ell}(-n) = \bfZ_{\ell}/\ell^M\bfZ_{\ell}(-n) = W[\ell^M],$$ where by $W[s]$ we mean the $s$-torsion.
For a place $v$ of $\bfQ$, and $M=V, W$ or $W_M$,  set $H^1_{\rm ur}(\bfQ_v, M) = \ker (H^1(\bfQ_v, M) \to H^1(I_v, M)).$ Then, following \cite{Rubin00}, section 1.3, we set  $$H^1_f(\bfQ_{v},V):=\begin{cases} H^1_{\rm ur}(\bfQ_v, V)  & \textup{if $v\neq \ell$}\\ \ker (H^1(\bfQ_{v}, V)\to H^1(\bfQ_{v}, V\otimes_{\bfQ_{v}} B_{\rm cris})) & \textup{if $v=\ell$}.\end{cases}$$ We define $H^1_f(\bfQ_v, W)$ as the image of $H^1_f(\bfQ_v, V)$ in $H^1(\bfQ_v, W)$. For the finite set $\Sigma$ of finite places of $\bfQ$, we then define the global Selmer group (cf. \cite{Rubin00}, Definition 1.5.1):
$$\mS^{\Sigma}(\bfQ, W):= \ker(H^1(\bfQ, W) \to \bigoplus_{v \not\in \Sigma}\frac{H^1(\bfQ_v, W)}{H^1_f(\bfQ_v, W)}).$$ One defines $\mS^{\Sigma}(\bfQ, W_M)$ similarly (cf. \cite{Rubin00}, p. 22). 
 \begin{lemma} \label{1} One has $\mS^{\Sigma}(\bfQ, W_M) = \mS^{\Sigma}(\bfQ, W)[\ell^M]$. \end{lemma}
\begin{proof} By Lemma 1.5.4 of \cite{Rubin00}, we get that there is a natural surjection of the left-hand side onto the right-hand side. However, the proof of that lemma uses the exact sequence in Lemma 1.2.2(i) in \cite{Rubin00} and in our case $W^{G_{\bfQ}}=0$, which shows that the surjection is in fact an isomorphism. \end{proof}

Let us first relate $\mS^{\Sigma}(\bfQ, W_m)$ to $\mS^{\emptyset}(\bfQ, W_m)$. 
\begin{lemma}\label{2} Suppose $n \neq 0$, $\val_{\ell}(\ell_1^{n+1} -1)=m-1$ and  $\val_{\ell}( \ell_i^{n+1}-1)=0$ for all $i=2,3, \dots, k$. Then one has $$\#\mS^{\Sigma}(\bfQ, W_m) \leq \ell^{m-1} \#\mS^{\emptyset}(\bfQ, W_m).$$\end{lemma}

\begin{proof} Fix $s \in \{1,2,\dots, k\}$. Since $W$ is unramified at $\ell_s$ we get $H^1_{\rm ur}(\bfQ_{\ell_s}, W) = H^1_f(\bfQ_{\ell_s}, W)$ (by Lemma 1.3.5(iv) in \cite{Rubin00})   as well as  $H^1_{\rm ur}(\bfQ_{\ell_s}, W_m) = H^1_f(\bfQ_{\ell_s}, W_m)$ (by Lemma 1.3.8(ii) in \cite{Rubin00}) and 
 \be\label{eq0}H^1(I_{\ell_s}, W_m)  = \Hom(\bfZ_{\ell}(1), W_m) =W_m(-1).\ee This gives an upper bound of $\ell^m$ on the order of the quotient $H^1(\bfQ_{\ell_s}, W_m)/H^1_f(\bfQ_{\ell_s}, W_m)$. However, let us now show that the upper bound is in fact $\ell^{m-1}$ (resp. $1$) if $s=1$ (resp. $s\neq 1$). Indeed, this will follow if we show that the map $H^1(\bfQ_{\ell_s}, W_m) \to H^1(I_{\ell_s}, W_m)$ is not surjective (resp. is the zero map) if $s=1$ (resp. $s\neq 1$). To do so consider the inflation-restriction sequence (where we set $G:=\Gal(\bfQ_{\ell_s}^{\rm ur}/\bfQ_{\ell_s})$):
$$H^1(G, W_m) \to H^1(\bfQ_{\ell_s}, W_m) \to H^1(I_{\ell_s}, W_m)^{G}\to H^2(G, W_m).$$ The last group in the above sequence is zero since $G\cong \hat{\bfZ}$ and $\hat{\bfZ}$ has cohomological dimension one. This means that the image of the restriction map $H^1(\bfQ_{\ell_s}, W_m) \to H^1(I_{\ell_s}, W_m)$ equals $H^1(I_{\ell_s}, W_m)^{G}$. Let us show that the  latter $\bfZ_{\ell}$-module is a proper submodule of $H^1(I_{\ell_s}, W_m)$  (resp. is the zero module) if $s=1$ (resp. $s\neq 1$). Indeed, $$ H^1(I_{\ell_s}, W_m)^{G}=\Hom_G(\bfZ_{\ell}(1), \frac{\frac{1}{\ell^m}\bfZ_{\ell}}{\bfZ_{\ell}}(-n)) = \Hom_G(\bfZ_{\ell}, \frac{\frac{1}{\ell^m}\bfZ_{\ell}}{\bfZ_{\ell}}(-n-1)).$$ So, $\phi \in H^1(I_{\ell_s}, W_m)$ lies in $H^1(I_{\ell_s}, W_m)^G=\Hom_G(\bfZ_{\ell}, \frac{\frac{1}{\ell^m}\bfZ_{\ell}}{\bfZ_{\ell}}(-n-1))$ if and only if $\phi(x)=\phi(gx) = g\cdot \phi(x)=\epsilon^{-n-1}(g)\phi(x)$ for every $x \in I_{\ell_s}$ and every $g \in G$, i.e., if and only if \be \label{eq1} (\epsilon^{-n-1}(g)-1)\phi(x)\in \bfZ_{\ell}\quad \textup{for every $x \in I_{\ell_s}$, $g \in G$.}\ee Since $\Frob_{\ell_s}$ topologically generates $G$, we see that \eqref{eq1} holds if and only if it holds for every $x \in I_{\ell_s}$ and for $g=\Frob_{\ell_s}$. We have $\epsilon^{-n-1}(\Frob_{\ell_s})-1=\ell_s^{-n-1}-1=\frac{1-\ell_s^{n+1}}{\ell_s^{n+1}}$. Since $\ell_s^{n+1}\in \bfZ_{\ell}^{\times}$, condition \eqref{eq1} becomes \be \label{eq2} (1-\ell_s^{n+1})\phi(x)\in \bfZ_{\ell}\quad \textup{for every $x \in I_{\ell_s}$.}\ee By our assumption $\val_{\ell}(1-\ell_s^{n+1})=m-1$ (resp. $\val_{\ell}(1-\ell_s^{n+1})=0$) if $s=1$ (resp. $s\neq 1$), which implies that \eqref{eq2} is equivalent to $\ell^{m-1}\phi(x)=0$ (resp. $\phi(x)=0$) in $W_m$ if $s=1$ (resp. $s\neq 1$). Using the isomorphism \eqref{eq0} we see that this implies that $H^1(I_{\ell_s}, W_m)^G$ is a proper $\bfZ_{\ell}$-submodule of $H^1(I_{\ell_s}, W_m)$ as $W_m(-1)$ certainly contains elements not annihilated by $\ell^{m-1}$. 

Now, by the Poitou-Tate duality (cf. \cite{Rubin00}, Theorem 1.7.3) we have an exact sequence $$0 \to \mS^{\emptyset}(\bfQ, W_m) \to \mS^{\Sigma}(\bfQ, W_m)\to \bigoplus_{i=1}^k\frac{H^1(\bfQ_{\ell_i}, W_m)}{H^1_{f}(\bfQ_{\ell_i}, W_m)}.$$ As shown above the order of the module on the right is bounded from above by $\ell^{m-1}$. This gives the desired inequality.    \end{proof}

Let us record here one consequence of the above proof.
\begin{lemma} \label{cycl} Suppose $\mS^{\emptyset}(\bfQ, W_m)=0$. Assume $n \neq 0$, $\val_{\ell}(\ell_1^{n+1} -1)=m-1$ and $\val_{\ell}( \ell_i^{n+1}-1)=0$ for all $i=2,3, \dots, k$. Then  $\mS^{\Sigma}(\bfQ, W_m)$ is a cyclic $\bfZ_{\ell}$-module, i.e., $\mS^{\Sigma}(\bfQ, W_m)\cong \bfZ/\ell^s $. Furthermore, $\dim_{\bfF_{\ell}} \mS^{\Sigma}(\bfQ, W_1)=1$.
\end{lemma} 
\begin{proof} From the Poitou-Tate duality (and the first isomorphism theorem for modules) we get $\mS^{\Sigma}(\bfQ, W_m) \subset \frac{H^1(\bfQ_{\ell_1}, W_m)}{H^1_{\rm ur}(\bfQ_{\ell_1}, W_m)}\cong I$,  where $I$ is the image of the restriction map $H^1(\bfQ_{\ell_1}, W_m) \to H^1(I_{\ell_1}, W_m)\cong W_m$. The last module is cyclic. The one-dimensionality statement follows   from this and Lemma \ref{1}. \end{proof}


From now on set $n=1$, so $W=\bfQ_{\ell}/\bfZ_{\ell}(-1)$. 
\begin{prop} \label{prop1} The Selmer group $\mS^{\emptyset}(\bfQ, W_m)$ is trivial. \end{prop} \begin{proof} It is enough to show that the group $\mS^{\emptyset}(\bfQ, W_1)$ is trivial. Indeed, Lemma \ref{1} shows $\mS^{\emptyset}(\bfQ, W_m)=\mS^{\emptyset}(\bfQ, W)[\ell^m]$. So it suffices to show that  $\mS^{\emptyset}(\bfQ, W)=0$. Since the latter module is divisible, it is enough to show that it has no $\ell$-torsion, i.e., that $\mS^{\emptyset}(\bfQ, W)[\ell]=\mS^{\emptyset}(\bfQ, W_1)=0$. 
It follows from Fontaine-Laffaille theory that $H^1_f(\bfQ_{\ell}, W_1)=H^1_{\rm ur}(\bfQ_{\ell}, W_1)$ so that $\mS^{\emptyset}(\bfQ, W_1) =  \Hom(\Cl_{\bfQ(\mu_{\ell})}, W_1)^{\Gal(\bfQ(\mu_{\ell})/\bfQ)}$. The latter module is zero by Herbrand's Theorem since the relevant Bernoulli number $B_2=1/6$ (see e.g., Theorem 6.17 in \cite{Washingtonbook}). \end{proof}

\begin{proof}[Proof of Proposition \ref{main}] Assume that $\rho_m$ as in the proposition exists. 
We can treat $\rho_m$ as an element  of $H^1(\bfQ, (\bfZ/\ell^m)(\chi_1\chi_2^{-1}))$ which is not annihilated by $\ell^{m-1}$ because its mod $\ell$ reduction is non-split. By Lemma \ref{4} we have $\chi_1=1$ and $\chi_2=\epsilon$. Also note that $\bfZ/\ell^m(\epsilon^{-1}) \cong W_m$. The extension given by $\rho_m$ being unramified away from $\Sigma'$ and  Fontaine-Laffaille (at $\ell$) in fact gives rise to an element  inside $\mS^{\Sigma}(\bfQ, W_m)\subset H^1(\bfQ, W_m)$ not annihilated by $\ell^{m-1}$.
  However, combining  Lemma \ref{2} applied in the case $n=1$ with Proposition \ref{prop1}
we see that $\mS^{\Sigma}(\bfQ, W_m)$ is annihilated by $\ell^{m-1}$ which leads to a contradiction.
\end{proof}

\begin{prop} \label{no inf} Let  $\rho': G_{\Sigma'} \to \GL_2(\bfF_{\ell}[X]/X^2)$ be a representation such that 
 \begin{itemize}
\item [(i)] $\rho'$ is Fontaine-Laffaille; 
\item [(ii)] $\det \rho'=\ov{\epsilon}$;
\item [(iii)] $\tr \rho' = \chi_1 + \chi_2$ for some Galois characters $\chi_1, \chi_2: G_{\Sigma'} \to (\bfF_{\ell}[X]/X^2)^{\times}$ with $\chi_1\equiv 1$ mod $X$ and $\chi_2\equiv \ov{\epsilon}$ mod $X$;
\item[(iv)]  $\rho' \equiv \ov{\rho}$ mod $X$.
\end{itemize}
Then $\rho'$ is isomorphic to $\ov{\rho}$ viewed as an $\bfF_{\ell}[X]/X^2[G_{\Sigma'}]$-module via the natural inclusion $\GL_2(\bfF_{\ell}) \hookrightarrow \GL_2(\bfF_{\ell}[X]/X^2)$.\end{prop}
\begin{proof} Using again the main theorem of \cite{Urban99} we conclude that $\rho'$ can be conjugated to a representation of the form $\bmat \chi_1&* \\ & \chi_2\emat.$ Hence $\chi_1$ and $\chi_2$ as subquotients of $\rho'$ are also Fontaine-Laffaille. Again arguing as in the proof of Proposition 9.5 in \cite{BergerKlosin13} we get that  $1$ and $\ov{\epsilon}$ do not admit any non-trivial infinitesimal Fontaine-Laffaille deformations, so we must have $\chi_1=1$ and $\chi_2=\ov{\epsilon}$. 
 This puts us in the setup of section 6 of \cite{BergerKlosin13} with Assumption 6(ii) satisfied. Hence the claim follows from Proposition 7.2 of \cite{BergerKlosin13}, using Lemma \ref{cycl} above to see that Assumption 6(i) is also satisfied. 
\end{proof}

Let $\mL$ be the category of local complete Noetherian $\bfZ_{\ell}$-algebras with residue field $\bfF_{\ell}$. Consider  deformations $\rho': G_{\Sigma'} \to \GL_2(A)$ of $\ov{\rho}$ for $A$ an object of $\mL$ which are such that:
\begin{itemize}
\item $\det \rho'=\epsilon$;
\item $\rho'$ is Fontaine-Laffaille at $\ell$.
\end{itemize}
Since $\ov{\rho}$ has scalar centralizer the above deformation problem is representable (cf. \cite{Ramakrishna93}, p. 270) by a universal deformation ring $R$. We write $\sigma: G_{\Sigma'} \to \GL_2(R)$ for the universal deformation. 

Let $I$ be the ideal of reducibility of the universal deformation $\sigma$, i.e., $I$ is the smallest ideal $I' \subset R$ such that $\tr \sigma$ is a sum of characters $\chi_1$ and $\chi_2$ mod $I'$ with the property that $\chi_1$ reduces to 1 and $\chi_2$ reduces to $\ov{\epsilon}$ modulo the maximal ideal $\fm_R$ of $R$. 
\begin{thm} \label{main3} Suppose $\val_{\ell}(\ell_1^2-1)=1$ and $\val_{\ell}(\ell_i^2-1)=0$ for all $i=2,3,\dots, k$. 
Then $I=\fm_R$. \end{thm}

\begin{proof} 
It follows from Proposition \ref{main} (and universality of $R$) that $R/I$ does not admit a surjection to $\bfZ/\ell^2$. Similarly it follows from Proposition \ref{no inf} that $R/I$ does not admit a surjection to $\bfF[X]/X^2$. Thus $I$ is the maximal ideal by Lemma 3.5 in \cite{Calegari06}.   \end{proof}
Let us explain one consequence of Theorem \ref{main3}. 
If $A$ is any object in $\mL$ and $\rho: G_{\Sigma'} \to \GL_2(A)$ is a continuous representation such that  \begin{itemize}
\item [(i)] $\rho$ is   Fontaine-Laffaille; 
\item [(ii)] $\det \rho=\epsilon$;
\item [(iii)] $\tr \rho = \chi_1 + \chi_2$ for some Galois characters $\chi_1, \chi_2: G_{\Sigma'} \to A^{\times}$ with $\chi_1\equiv 1$ mod $\fm_A$ and $\chi_2\equiv \ov{\epsilon}$ mod $\fm_A$;
\item[(iv)]  $\rho = \ov{\rho}$ mod $\fm_A$,
\end{itemize} then the $\bfZ_{\ell}$-algebra map $\phi: R\to A$ whose existence follows from universality of $R$ factors through (by the definition of $I$) a $\bfZ_{\ell}$-algebra map $R\twoheadrightarrow \bfF_{\ell}=R/I \xrightarrow{\phi} A$ such that $\rho$ is isomorphic to $\ov{\rho}$ viewed   as a $A[G_{\Sigma'}]$-module via $\phi$.

\begin{cor} \label{main4} Let $k=2$. Suppose $\val_{\ell}(\ell_1^2-1)=1$ and $\val_{\ell}(\ell_2^2-1)=0$. Let $\bfT$ be the Hecke algebra as in section \ref{Intro} and $\fm$ a maximal Eisenstein ideal as in Theorem \ref{thmintro}. Then there does not exist a Galois representation $\rho: G_{\Sigma'} \to \GL_2(\bfT/\fm^2)$ such that $\rho$ satisfies (i)-(iv) as above with $A=\bfT/\fm^2$. \end{cor}
\begin{proof} Suppose $\rho$ as in the statement exists. Note that $\bfT/\fm^2$ is an object of $\mL$. Then by universality of $R$ we get a $\bfZ_{\ell}$-algebra map $\phi: R \to \bfT/\fm^2$. Let us first see that this map is surjective. Indeed, viewing $\bfT$ as the Hecke algebra acting on the space of weight 2 cusp forms of level $\Gamma_0(\ell_1\ell_2)$ we first complete it at the ideal $\fm$ and note that $\bfT_{\fm}$ is an element of $\mL$ (since $\bfT_{\fm}/\fm \bfT_{\fm} = \bfT/\fm \bfT=\bfF_{\ell}$). For every minimal prime $\fP$ of $\bfT_{\fm}$ we have a canonical map $\bfT_{\fm} \twoheadrightarrow \bfT_{\fm}/\fP$ given by sending operators $T_{r}$ and $U_r$ to the eigenvalues of the corresponding cusp form. It follows  from Proposition A.2.3 and A.2.2(2) in \cite{WakeWangErickson18preprint} that the algebra $\bfT_{\fm}$ is generated by   the operators $T_r$ for $r \nmid \ell\ell_1\ell_2$. Indeed, our assumptions on the valuations of the $\ell_i$ imply that the Atkin-Lehner signature denoted in  \cite{WakeWangErickson18preprint} by $\epsilon$ equals $(-1,1)$ - this is forced by the condition that the constant term of the relevant Eisenstein series (cf.  equation (1.3.1) in \cite{WakeWangErickson18preprint}) vanishes modulo $\ell$. In other words our Hecke algebra $\bfT_{\fm}$ equals the Hecke algebra denoted in  \cite{WakeWangErickson18preprint} by $\mathbb{T}_U^{(-1,1),0}$, which in turn equals $\mathbb{T}^{(-1,1), 0}$ by Proposition A.2.3 in  \cite{WakeWangErickson18preprint}. It then follows from Proposition A.2.2 that this last Hecke algebra is generated by $T_r$ for $r \nmid \ell\ell_1\ell_2$. Thus the intersection of all the minimal primes  $\bigcap_{\fP} \fP$ equals 0 as it consists of all the operators $T$ such that $Tf=0$ for all eigenforms $f$ of $\bfT_{\fm}$. Hence in particular $\bfT_{\fm}$ injects into $\prod_{\fP} \bfT_{\fm}/\fP=\tilde{\bfT}_{\fm}$, where $\tilde{\bfT}_{\fm}$ is the normalization of $\bfT_{\fm}$. 

We claim that the combined map $R \to \prod_{\fP} \bfT_{\fm}/\fP=\tilde{\bfT}_{\fm} \supset \bfT_{\fm}$ surjects onto $\bfT_{\fm}$.  This is a standard argument, which we summarize here in our situation.
 First arguing as in the proof of Proposition 7.13 in \cite{BergerKlosin13} using Theorem \ref{main3} above for the cyclicity of $R/I$ we conclude that $R$ is generated by the set $\{\tr \sigma(\Frob_r) \mid r \not \in \Sigma'\}$. Since each of these traces is mapped to $T_r$ under the map $R \to \tilde{\bfT}_{\fm}$ we see that the image is contained in $\bfT_{\fm}$. In fact, it equals $\bfT_{\fm}$ as we showed above that $\bfT_{\fm}$ is generated by $T_{r}$ with $r \nmid \ell\ell_1\ell_2$. 

Having established the surjectivity of $R \to \bfT_{\fm}$ we now  use (iii) above and the definition of $I$ to conclude that the induced surjection $\phi: R \twoheadrightarrow \bfT/\fm^2$ factors through a $\bfZ_{\ell}$-algebra map $R\twoheadrightarrow R/I \xrightarrow{\phi} \bfT/\fm^2$. However, $R/I \cong \bfF_{\ell}$ by Theorem \ref{main3} implying that $\fm=\fm^2$, which is absurd. 
\end{proof}



\section{Character groups of $J_0(65)$ as Hecke modules}\label{Matrices} 

In this section $J:=J_0(65)$. In this case, $J=J^{\rm new}$. 
Let $M_p$ denote the character group of $J$ at $p$ as defined in the introduction. 
For $p=5, 13$, $M_p$ is a free abelian group of rank $\dim(J)=5$. 
By the N\'eron mapping property, 
the action of the Hecke algebra $\bfT$ on $J$ extends canonically to an action on the N\'eron model $\mJ$ of $J$ over $\bfZ_p$. For $p=5, 13$, 
$\bfT$ acts faithfully on $\mJ_{\bfF_p}^0$, and hence 
also on $M_p$ (because $J$ has purely toric reduction at $p$). The main result of this section is the fact that 
$M_5$ and $M_{13}$ are isomorphic as $\T$-modules. The proof is based on explicit calculations with Brandt matrices; cf. \cite{Gross87}. 

\begin{rem}\label{rem3.1}
The algebra $\bfT\otimes\bfQ$ is semi-simple of dimension $5$ over $\bfQ$. 
Since $\bfT\otimes\bfQ$ acts faithfully on $M_p\otimes\bfQ$, $p=5, 13$, which is also $5$-dimensional over $\bfQ$, one easily concludes 
that $M_p\otimes\bfQ$ is free over $\bfT\otimes\bfQ$ of rank $1$. Thus, $M_5\otimes \bfQ\cong M_{13}\otimes \bfQ$ as $\T$-modules, 
but the isomorphism over $\bfZ$ is more subtle. 
\end{rem}

\begin{prop}\label{prop3.2} There are isomorphisms of $\bfT$-modules 
$M_5\cong M_{13}\cong \bfT$. 
\end{prop}
\begin{proof}
The following \texttt{Magma} routine computes the action of $T_n$ on $M_5$ for a given positive integer $n$:\\
\texttt{
> B5:= BrandtModule(5, 13); \\
> M5:= CuspidalSubspace(B);\\
> Sn:=HeckeOperator(M5, n);\\
}
The result is an explicit matrix $S_n\in M_5(\bfZ)$. Repeating the same process with the roles of $5$ and $13$ interchanged, we 
get another matrix $S_n'\in M_5(\bfZ)$ by which $T_n$ acts on $M_{13}$ (with respect to implicit $\bfZ$-bases chosen by the program). 

A calculation with discriminants shows that 
$\bfT$, as a free $\bfZ$-module of rank $5$, is generated by the Hecke operators $T_1, T_2, T_3, T_5, T_{11}$; cf. \cite[Sec. 3]{KlosinPapikian18}. 
We have 
$$
S_1=\begin{bmatrix*}[r]
1 & 0 & 0 & 0 & 0 \\ 
0 & 1 & 0 & 0 & 0 \\ 
0 & 0 & 1 & 0 & 0 \\ 
0 & 0 & 0 & 1 & 0 \\ 
0 & 0 & 0 & 0 & 1
\end{bmatrix*}, \quad 
S_2=\begin{bmatrix*}[r]
-1 & -1 & 1 & 0 & 0 \\ 
-1 & -1 & 0 & 1 & 0\\
2 & -1 & 0 & 0 & -1\\
-1 & 2 & 0 & 0 & -1\\
0 & 0 & 0 & 0 &-1
\end{bmatrix*}, 
$$
$$
S_3=\begin{bmatrix*}[r]
0 &-1 & 0 & 1 & -1 \\
-1 & 0 & 1 & 0 & -1\\
-1 & 2 & 1 & 0 & -2\\
2 & -1 & 0 & 1 & -2\\ 
0 & 0 & 0 & 0 & -2
\end{bmatrix*}, \quad 
S_5=\begin{bmatrix*}[r]
0 & 1 & 0 & 0 & -1 \\ 
1 & 0 & 0 & 0 & -1\\
0 & 0 & 0 & 1 & -1\\
0 & 0 & 1 & 0 & -1\\
0 & 0 & 0 & 0 & -1
\end{bmatrix*}, 
$$
$$
S_{11}=\begin{bmatrix*}[r]
0 & 3 & 0 & -1 & 0\\
3 & 0 & -1 & 0 & 0 \\
1 & -2 & -1 & 2 & 1 \\
-2 & 1 & 2 & -1 & 1\\
0 & 0 & 0 & 0 & 2
\end{bmatrix*}.
$$
In \texttt{Magma}, the action of Hecke operators on $M_p$ is defined to be from the right, i.e., as on row vectors. 
Let $v=[1, 0, 0,0,0] \in M_5$, and 
$$
A:=\begin{bmatrix*}
vS_1\\
vS_2\\
vS_3\\
vS_5\\
vS_{11}
\end{bmatrix*}
=
\begin{bmatrix*}[r]
1 & 0 & 0 & 0 & 0\\
-1 & -1 & 1 & 0 & 0\\
0 & -1 & 0 & 1 & -1\\
0 & 1 & 0 & 0 & -1\\
0 & 3 & 0 & -1 & 0
\end{bmatrix*}. 
$$
One easily verifies that $\det(A)=1$, hence 
$$
M_5=\bfZ vS_1+\bfZ vS_2+\bfZ vS_3+\bfZ vS_5+\bfZ vS_{11}= v\bfT. 
$$
Thus, $M_5\cong \bfT$ is a free $\bfT$-module of rank $1$. A similar calculation with $M_{13}$, gives 
$$
A':=\begin{bmatrix*}
vS_1'\\
vS_2'\\
vS_3'\\
vS_5'\\
vS_{11}'
\end{bmatrix*}
=
\begin{bmatrix*}[r]
1& 0 & 0 & 0 & 0\\
0 & 2 & 0 & -1 & -1\\
-1 & 0 & 1 & 0 & -1\\
0 & -1 & 0 & 0 & 0\\
1 & -2 & -1 & 0 & 2
\end{bmatrix*}. 
$$
In this case, $\det(A)=-1$, hence again $M_{13}=v \bfT$. 
\end{proof}

\begin{rem}
The fact that $M_5$ and $M_{13}$ are free $\bfT$-modules is a coincidence (a priori, we don't see a reason for this to happen). 
To emphasize this point, we note that the dual $M_5^\ast=\Hom(M_5, \bfZ)$ of $M_5$ 
with induces action of $\T$ is not a free $\bfT$-module. 
(On $M_5^\ast$ the Hecke operator $T_n$ acts by the transpose of the matrix by which it acts on $M_5$). 
Indeed, otherwise we get $\bfT\cong \Hom(\bfT, \bfZ)$, which implies that the localization of $\bfT$ at any maximal ideal is Gorenstein 
in contradiction to \cite[Prop. 3.7]{KlosinPapikian18}. 
\end{rem}

\begin{rem} The proof of Proposition \ref{prop3.2} is rather ad hoc. 
Suppose more generally that we are given two $\bfT$-modules $M, M'$ for a Hecke algebra of some level $N$ such that $M, M'$ 
are free of the same finite rank over $\bfZ$ and 
$M\otimes_\bfZ\bfQ\cong_{\bfT} M'\otimes_\bfZ\bfQ$. Also, suppose we are 
able to compute efficiently the matrices $S_n, S_n'$ by which $T_n$ acts on $M$ and $M'$, respectively. 
The question of the integral isomorphism $M\cong_{\bfT(N)} M'$ is equivalent to the existence of 
an invertible matrix $S\in \GL_r(\bfZ)$ such that 
$S S_nS^{-1}=S_n'$ for all $n\geq 1$; here $r=\mathrm{rank}_\bfZ(M)$. In fact, 
it is enough to find such $S$ that works for all $n$ up to an explicit bound depending on $N$ (the Sturm bound). 
Despite the elementary nature of this question, computationally it is challenging. 
The problem of integral conjugacy of matrices is a classical problem related to class groups of orders in number fields (see \cite{LatimerMacDuffee33}), 
and there are algorithms that solve this problem (see \cite{Sarkisjan79}, \cite{Grunewald80}), but 
currently these algorithms do not seem to be implemented in any of the standard computational programs, such as \texttt{Magma}. 
(Given two $m\times m$ matrices $A$ and $B$ with rational or integral entries, 
\texttt{Magma} currently can test whether $A$ is conjugate to $B$ in $\GL_m(\bfZ)$ 
only if $m=2$.)
\end{rem}

\subsection*{Acknowledgements} We are very grateful to Ken Ribet for suggesting that his construction in \cite{Ribet90Israel} 
might lead to a counterexample to Ogg's conjecture, and for other helpful suggestions about the exposition in an earlier version of this paper. 
We thank Hwajong Yoo for pointing out several misstatements in earlier versions of this paper, 
and for directing us to the reference \cite{YooTAMS}. 
The second author is also grateful to Fu-Tsun Wei for useful discussions related to the topic of this paper.

\bibliographystyle{amsalpha}
\bibliography{RibetOgg}
\end{document}